\documentclass[12pt]{article}
\usepackage{amsfonts,epsf,amsmath,amssymb}
\usepackage{amsthm}

\author{
Jernej Azarija \\
Institute of Mathematics, Physics and Mechanics \\
Jadranska 19, 1000 Ljubljana, Slovenia \\
jernej.azarija@gmail.com}
\title{Sharp upper and lower bounds on the number of spanning trees in Cartesian product of graphs}

\def\trimx#1{\ignorespaces#1\unskip} 
\newcommand{\bibitemart}[7]{\bibitem{#1} \trimx{#2}, {\it \trimx{#3}}\/, \trimx{#4} {\bf \trimx{#5}} (\trimx{#6}) \trimx{#7}.}
\newcommand{\bibitembook}[5]{\bibitem{#1} \trimx{#2}, \trimx{#3} (\trimx{#4}, \trimx{#5}).}

\newtheorem{theorem}{Theorem}

\newtheorem{cor}{Corollary}

\begin{document}
\maketitle
\begin{abstract}
Let $G_1$ and $G_2$ be simple graphs and let $n_1 = |V(G_1)|$, $m_1 = |E(G_1)|$, $n_2 = |V(G_2)|$ and $m_2 = |E(G_2)|.$ In this paper we derive sharp upper and lower bounds for the number of spanning trees $\tau$ in the Cartesian product $G_1 \square G_2$ of $G_1$ and $G_2$. We show that: $$ \tau(G_1 \square G_2) \geq \frac{2^{(n_1-1)(n_2-1)}}{n_1n_2} (\tau(G_1) n_1)^{\frac{n_2+1}{2}} (\tau(G_2)n_2)^{\frac{n_1+1}{2}}$$ and $$\tau(G_1 \square G_2) \leq \tau(G_1)\tau(G_2) \left[ \frac{2m_1}{n_1-1} + \frac{2m_2}{n_2-1} \right]^{(n_1-1)(n_2-1)}.$$ We also characterize the graphs for which equality holds. As a by-product we derive a formula for the number of spanning trees in $K_{n_1} \square K_{n_2}$ which turns out to be $n_{1}^{n_1-2}n_2^{n_2-2}(n_1+n_2)^{(n_1-1)(n_2-1)}.$
\end{abstract}

\section{Introduction}
An important invariant in graph theory is $\tau(G)$, the number of spanning trees of a graph $G$. The first result related to $\tau(G)$ dates back to 1847 and is attributed to Kirchhoff \cite{Kirchhoff}. In his celebrated theorem he has shown that the number of spanning trees of a graph $G$ is closely related to the cofactor of a special matrix (the {\em Laplacian matrix}) that can be obtained after substracting the {\em adjacency matrix} from the respective {\em degree matrix} (a diagonal matrix with vertex degrees on the diagonals). If by $Q(G)$ we denote the Laplacian matrix of a graph $G$ of order $n$ with eigenvalues $0 = \lambda_1 \leq \cdots \leq \lambda_n$, then a corollary of Kirchhoff's theorem can be stated as \begin{equation}\label{form} \tau(G) = \frac{1}{n} \lambda_2 \cdots \lambda_n. \end{equation} For example, as the eigenvalue $n$ of $Q(K_n)$ has multiplicity $(n-1)$, it follows that \begin{equation}\label{Cayley} \tau(K_n) = n^{n-2}. \end{equation} Equation (\ref{Cayley}) is also refered to as {\em Cayley formula} as a tribute to its discoverer Arthur Cayley \cite{Cayley}. For a survey of known results related to the Laplacian spectrum of graphs we refer the reader to \cite{Mohar}. 

Since the result of Cayley, many interesting identites for the number of spanning trees for various classes of graphs have been derived. For example, Bogdanowicz \cite{Bog} showed that the number of spanning trees of the $n$-fan $F_{n+1}$ equals to $f_{2n}$ where $f_n$ is the $n$'th {\em Fibonacci number}.  A similar result relating the number of spanning trees of the wheel graph to {\em Lucas numbers} is also known \cite{IEE}. Counting the number of spanning trees is not only an area that is rich with surprising identities but also holds a fundamential role in other scientific areas such as physics \cite{Phy3,Phy1} networking theory \cite{Net1} and also finds applications in the study of various electrical networks \cite{El}. Since {\em graph products} (as defined in \cite{Klavzar}) form a basis for many network topologies it is natural to study the function $\tau$ in relation with various graph products.

In this paper we study the number of spanning trees in the Cartesian product of graphs. For simple graphs $G_1$ and $G_2$, the Cartesian product $G_1 \square G_2$ is defined as the graph with vertex set $V(G_1) \times V(G_2)$ such that two vertices $(u,u')$ and $(v,v')$ are adjacent if and only if either $u = v$ and $u'$ is adjacent to $v'$ in $G_2$, or $u' = v'$ and $u$ is adjacent to $v$ in $G_1$. 

In what follows $G_1$ and $G_2$ will denote simple graphs of order $n_1$ and $n_2$ such that $m_1 = |E(G_1)|$ and $m_2 = |E(G_2)|.$ Moreover, we will denote by $\lambda_1, \ldots, \lambda_{n_1}$ and $\mu_1, \ldots, \mu_{n_2}$ the eigenvalues of $Q(G_1)$ and $Q(G_2)$ respectively. Using this notation, we can state the well know (see \cite{Mohar} for a survey of results related to the Laplacian spectrum) fact relating the eigenvalues of $G_1$ and $G_2$ to the eigenvalues of $G_1\square G_2$ which are $$ \lambda_i + \mu_j \quad \hbox{for} \quad i = 1, \ldots, n_1 \quad \hbox{and} \quad j = 1, \ldots, n_2.$$ Applying the later equality to identity (\ref{form}) and using the fact that $\lambda_1 = \mu_1 = 0$ one obtains the following formula for the number of spanning trees for the Cartesian product of $G_1$ and $G_2$: \begin{equation}\label{SPProd} \tau(G_1\square G_2) = \tau(G_1)\tau(G_2)\prod_{i = 2}^{n_1} \prod_{j=2}^{n_2} (\lambda_i + \mu_j). \end{equation} 

\section{Upper and lower bounds for $\tau(G_1 \square G_2)$}

We are going to simplify equation (\ref{SPProd}) as to obtain upper and lower bounds for $\tau(G_1 \square G_2).$ Furthermore we will characterize the graphs for which equality holds and derive a formula for the number of spanning trees of the {\em Rook's graph} $K_{n_1} \square K_{n_2}.$

\begin{theorem}\label{T1}
$\tau(G_1 \square G_2) \geq \frac{2^{(n_1-1)(n_2-1)}}{n_1n_2} (\tau(G_1) n_1)^{\frac{n_2+1}{2}} (\tau(G_2)n_2)^{\frac{n_1+1}{2}}$ where equality holds if and only if $G_1$ or $G_2$ is not connected or $n_1 = n_2$ and $G_1 \simeq G_2 \simeq K_{n_1}.$
\end{theorem}

\begin{proof}
Consider the expression: $$ \prod_{i=2}^{n_1} \prod_{j = 2}^{n_2} (\lambda_i+\mu_j).$$ By the inequality of arithmetic and geometric means $\lambda_i + \mu_j \geq 2\sqrt{\lambda_i \mu_j}$ for every $i,j,$ it therefore follows that $$ \prod_{i=2}^{n_1} \prod_{j = 2}^{n_2} (\lambda_i+\mu_j) \geq \prod_{i=2}^{n_1} \prod_{j = 2}^{n_2} 2\sqrt{\lambda_i \mu_j} = 2^{(n_1-1)(n_2-1)} \prod_{i=2}^{n_1} \prod_{j = 2}^{n_2} \sqrt{\lambda_i \mu_j}.$$ The last expression can also be writen as: $$2^{(n_1-1)(n_2-1)} \prod_{i=2}^{n_1} \sqrt{\lambda_i^{n_2-1}} \prod_{j=2}^{n_2} \sqrt{\mu_j^{n_1-1}}. $$ We now multiply and divide the last expression by $\sqrt{{n_1}^{n_2-1}{n_2}^{n_2-1}}$ and obtain: $$2^{(n_1-1)(n_2-1)}\sqrt{{n_1}^{n_2-1} {n_2}^{n_1-1}} \frac{ \prod_{i=2}^{n_1} \sqrt{\lambda_i^{n_2-1}}}{\sqrt{{n_1}^{n_2-1}}} \frac{\prod_{j=2}^{n_2} \sqrt{\mu_j^{n_1-1}}}{\sqrt{{n_2}^{n_1-1}}}$$ which, according to (\ref{form}), equals  $$2^{(n_1-1)(n_2-1)} (\tau(G_1)n_1)^{\frac{n_2-1}{2}} (\tau(G_2)n_2)^{\frac{n_1-1}{2}}\,. $$ The stated inequality now follows after combining the derived result with equation (\ref{SPProd}). 

We now examine the cases in which equality holds. If $G_1$ or $G_2$ is not connected, then equality clearly holds as $\tau(G_1\square G_2) = 0.$ Therefore, let us assume $G_1$ and $G_2$ are connected.  As we derived our inequality using the inequality of arithmetic and geometric means it follows that equality holds if and only if $\lambda_i = \mu_j$ for every $i = 2, \ldots, n_1$ and $j = 2, \ldots, n_2.$ The later holds if and only if $$\lambda_2 = \cdots = \lambda_{n_1} = \mu_2 = \cdots = \mu_{n_2},$$ which means that $Q(G_1)$ and $Q(G_2)$ have eigenvalues of multiplicity $n_1-1$ and $n_2-1$, respectively. As the only graph of order $k$ whose Laplacian matrix has an eigenvalue of multiplicity $k-1$ is $K_k$,  it follows, that $n_1 = n_2$ and thus $G_1 \simeq K_{n_1} \simeq  G_2.$
\end{proof}

In the proof of Theorem \ref{T1} we applied the inequality of arithmetic and geometric means to each summand of (\ref{SPProd}) individually. Observe that the same inequality can be applied to the factors of equation (\ref{SPProd}). In the next theorem we use this observation and the fact that $ \sum_{i=2}^{n_1} \lambda_i = 2m_1$ and $\sum_{i=2}^{n_2} \mu_i = 2m_2$ in order to derive an upper bound for $\tau(G_1 \square G_2)$.

\begin{theorem}\label{T2}
$ \tau(G_1\square G_2) \leq \tau(G_1)\tau(G_2)\left[ \frac{2m_1}{n_1-1} + \frac{2m_2}{n_2-1} \right]^{(n_1-1)(n_2-1)},$ where equality holds if and only if $G_1$ or $G_2$ is not connected or $G_1 \simeq K_{n_1}$ and $G_2 \simeq K_{n_2}.$
\end{theorem}

\begin{proof}
As observed, we can bound equation (\ref{SPProd}) by applying the inequality of geometric and arithmetic means on its factors. We then obtain $$ \tau(G_1\square G_2) = \tau(G_1)\tau(G_2)\prod_{i = 2}^{n_1} \prod_{j=2}^{n_2} (\lambda_i + \mu_j)  \leq \tau(G_1)\tau(G_2)\left[\frac{\sum_{i=2}^{n_1} \sum_{j=2}^{n_2} (\lambda_i + \mu_j)}{(n_1-1)(n_2-1)}\right]^{(n_1-1)(n_2-1)}, $$ which we further simplify to $$ \tau(G_1)\tau(G_2) \left[ \frac{(n_2-1)\sum_{i=2}^{n_1} \lambda_i + (n_1-1)\sum_{j=2}^{n_2} \mu_j}{(n_1-1)(n_2-1)} \right]^{(n_1-1)(n_2-1)}. $$ Applying the identity for the summation of the eigenvalues of the Laplacian matrix we obtain $$ \tau(G_1)\tau(G_2)\left[\frac{2m_1}{n_1-1}+\frac{2m_2}{n_2-1}\right]^{(n_1-1)(n_2-1)},$$ which is what we wanted to show. 

Observe now, that if $G_1$ or $G_2$ is not connected, equality in the stated bound clearly holds. Thus, let us assume $G_1$ and $G_2$ are connected. Equality will then hold if and only if $$ \lambda_i + \mu_j = \lambda_{i'}+\mu_{j'} \quad \hbox{ for} \quad i,i' = 1,\ldots,n_1 \quad \hbox{and} \quad j,j' = 1,\ldots,n_2.$$ The later holding if and only if $$\lambda_2 = \cdots = \lambda_{n_1} \quad \hbox{and} \quad \mu_2 = \cdots = \mu_{n_2},$$ which means $G_1 \simeq K_{n_1}$ and $G_2 \simeq K_{n_2}$ as these are the only graphs of order $n_1$ and $n_2$ having eigenvalues of multiplicity $n_1-1$ and $n_2-1$, respectively.
\end{proof}

The statements of Theorems \ref{T1} and \ref{T2} simplify substantialy if $G_1$ and $G_2$ are trees. In this case we can write the implications of Theorem \ref{T1} and Theorem \ref{T2} as the following corollary.

\begin{cor}
If $G_1$ and $G_2$ are trees of order $n_1\geq 3$ and $n_2 \geq 3$ respectively, then $$ 2^{(n_1-1)(n_2-1)}{n_1}^{\frac{n_2-1}{2}} {n_2}^{\frac{n_1-1}{2}} < \tau(G_1 \square G_2) < 2^{2(n_1-1)(n_2-1)}.$$
\end{cor}

As we saw in Theorem \ref{T2}, the derived bound for $\tau(G_1 \square G_2)$ is tight whenever $G_1 \simeq K_{n_1}$ and $G_2 \simeq K_{n_2}.$ This, in combination with equation (\ref{Cayley}), readily gives an exact formula for the number of spanning trees of $K_{n_1} \square K_{n_2}$:

\begin{cor}
$\tau(K_{n_1} \square K_{n_2}) = {n_1}^{n_1-2}{n_2}^{n_2-2}(n_1+n_2)^{(n_1-1)(n_2-1)}.$
\end{cor}

Observe, that the same argument as used in Theorems \ref{T1} and \ref{T2} could be applied to the other standard graph products provided that a similar characterisation of their Laplacian spectrum is known. At present no result of this type was known to the author, hence we leave it as future work to investigate upper and lower bounds for the other graph products.

\section{Acknowledgements}

The author is thankful to Sandi Klav\v{z}ar for constructive discussions related to the problem.

\end{document}